\documentclass{amsart}
\usepackage{amssymb, amsmath, verbatim}
\usepackage[all]{xy}

\theoremstyle{plain}
\newtheorem{lem}{Lemma}[section]
\newtheorem{thm}[lem]{Theorem}
\newtheorem{cor}[lem]{Corollary}

\newtheorem{prop}[lem]{Proposition}

\theoremstyle{definition}
\newtheorem{defn}[lem]{Definition}
\newtheorem{example}[lem]{Example}

\theoremstyle{remark}

\newtheorem{question}[lem]{Question}
\newtheorem{problem}[lem]{Problem}
\newtheorem{rmrk}[lem]{Remark}
\newtheorem*{claim}{Claim}

\numberwithin{equation}{section}

\newcommand{\mc}[1]{\mathcal{#1}}

\newcommand{\al}{\aleph}
\newcommand{\af}{\alpha}
\newcommand{\bt}{\beta}

\newcommand{\om}{\omega}
\newcommand{\ka}{\kappa}
\newcommand{\lm}{\lambda}
\newcommand{\oml}{\om_1}
\newcommand{\alo}{\al_0}

\newcommand{\inv}[1]{#1^{-1}}
\newcommand{\vn}{\varnothing}

\newcommand{\istst}{it suffices to show that}

\newcommand{\istf}{it suffices to find}

\newcommand{\Wma}{We may assume}
\newcommand{\suchthat}{\mathrel{}|\mathrel{}}

\DeclareMathOperator{\cf}{cf}
\DeclareMathOperator{\dom}{dom}

\DeclareMathOperator{\ran}{ran}
\DeclareMathOperator{\powset}{\mc{P}}

\newcommand{\gen}[2][]{\left\langle #2\right\rangle_{#1}}

\newcommand{\imp}{\Rightarrow}
\newcommand{\biimp}{\Leftrightarrow}

\DeclareFontFamily{U}{mathb}{\hyphenchar\font45}
\DeclareFontShape{U}{mathb}{m}{n}{
      <5> <6> <7> <8> <9> <10> gen * mathb
      <10.95> mathb10 <12> <14.4> <17.28> <20.74> <24.88> mathb12
      }{}
\DeclareSymbolFont{mathb}{U}{mathb}{m}{n}
\DeclareFontSubstitution{U}{mathb}{m}{n}
\DeclareFontFamily{U}{mathx}{\hyphenchar\font45}
\DeclareFontShape{U}{mathx}{m}{n}{
      <5> <6> <7> <8> <9> <10>
      <10.95> <12> <14.4> <17.28> <20.74> <24.88>
      mathx10
      }{}
\DeclareSymbolFont{mathx}{U}{mathx}{m}{n}
\DeclareFontSubstitution{U}{mathx}{m}{n}

\DeclareMathSymbol{\bigboxplus}{1}{mathx}{"D0}
\DeclareMathSymbol{\bigboxtimes}{1}{mathx}{"D2}
\DeclareMathAccent{\widecheck}{0}{mathx}{"71}

\newcommand{\bigcomm}{
  \mathop{
    \vcenter{
      \hbox{\oalign{\noalign{\kern-.3ex}\hfil$\vert$\hfil\cr
        \noalign{\kern-.7ex}
        $\smile$\cr\noalign{\kern-.3ex}}}
    }
  }  
}

\newcommand{\tlex}{{\mathrm{lex}}}

\newcommand{\closure}[1]{\overline{#1}}

\newcommand{\nbd}{\nobreakdash}

\newcommand{\cardd}{\mathfrak{d}}
\newcommand{\cardc}{\mathfrak{c}}

\newcommand{\betaoo}{\beta\om\setminus\om}

\newcommand{\character}[1]{\chi(#1)}

\newcommand{\wolog}{without loss of generality}

\newcommand{\cardcp}{\cardc^+}
\newcommand{\arhan}{Arhan\-gel${}^\prime$ski\u\i}

\newcommand{\locsplit}[3][.]{\mathrm{split}_{#2}^{\ifx#1.{}\else{#1}\fi}(#3)}

\newcommand{\dfn}{\begin{defn}}
\newcommand{\nfd}{\end{defn}}

\newcommand{\Ists}{It suffices to show that}

\newcommand{\pin}{\equiv_{\text{p}}}
\newcommand{\eqtuk}{\equiv_{\mathrm{T}}}
\newcommand{\leqtuk}{\leq_{\mathrm{T}}}
\newcommand{\leqrk}{\leq_{\mathrm{RK}}}

\newcommand{\abs}[1]{\lvert #1\rvert}

\begin{document}

\title{Between homeomorphism type and Tukey type}
\author{David Milovich}
\email{david.milovich@welkinsciences.com}
\urladdr{http://dkmj.org}
\address{
  Welkin Sciences\\
  2 N Nevada Ave, Suite 1280\\
  Colorado Springs, CO 80903, USA
}
\date{\today}
\begin{abstract}
Call a compact space $X$ {\it pin homogeneous} if
every two points $a,b$ are {\it pin equivalent},
meaning that there exists
a compact space $Y$, a quotient map $f\colon Y\to X$,
and a homeomorphism $g\colon Y\to Y$
such that $g\inv{f}\{a\}=\inv{f}\{b\}$.
We will prove a representation theorem for pin
equivalence; transitivity of pin equivalence
will be a corollary.

Pin homogeneity is strictly weaker than homogeneity
and pin equivalence is strictly stronger than
Tukey equivalence.
Just as with topological homogeneity,
no infinite compact $F$-space is pin homogeneous.
On the other hand, $X\times 2^{\character{X}}$
is pin homogeneous for every compact $X$.
And there is a compact pin homogeneous space
with points of different $\pi$-character.
\end{abstract}


\maketitle

\section{Introduction}
In this paper, all spaces are assumed to be Hausdorff.

Products, but apparently not much else,
preserve both compactness and homogeneity\footnote{
  A space $X$ is {\it homogeneous} if for each pair $(a,b)\in X^2$
  some autohomeomorphism $f\colon X\to X$ sends $a$ to $b$.}
of topological spaces.
Meanwhile, compact topological groups are ccc and \v Cech-Stone
remainders of infinite discrete spaces are not homogeneous.
For essentially these reasons,
``large'' homogeneous compact spaces are hard to come by.
Van Douwen's Problem, asked no later than 1980
and still open in all models of ZFC~\cite{vanmillsurvey},
asks whether there is a homogeneous compact space
with $\cardcp$-many disjoint open sets.
This a special case of a very natural question:

\begin{question}\label{bigquestion}
  Is every compact space $X$ a continuous image
  of some homogeneous compact space
  $Y$?~\cite{vanmillopit2}
\end{question}

Even for some important spaces $X$
without $\cardcp$-many disjoint open sets,
including $\oml+1$, $\beta\om$, and $\betaoo$,
the above question is open as far as I know.
Two significant partial results are:
\begin{enumerate}
\item (Motorov~\cite{motorov})
  If $X$ is first countable, compact,\footnote{
  Dow and Pearl~\cite{dowpearl} showed that
  compactness is not needed here.}
  and zero-dimensional, then $X^\omega$ is homogeneous.
\item (Kunen~\cite{kunenlhc}) No product of one or more
  infinite compact $F$-spaces and zero or more
  spaces with character less than $\cardc$ is
  homogeneous.
\end{enumerate}

Approaching Question~\ref{bigquestion} less directly,
we can consider weaker forms of homogeneity.
For example, say that a space is {\it Tukey homogeneous}
if every two points have Tukey equivalent neighborhood filters.
This is a much weaker than homogeneity.
For example, the compact space $2^\om\times 2^{\oml}_{\tlex}$
is Tukey homogeneous yet has points with different $\pi$-characters.
And, in striking contrast to Kunen's theorem,
$X\times 2^{\character{X}}$ is Tukey homogeneous
for every infinite space $X$.~\cite{mnth}
And though $\betaoo$ is Tukey inhomogeneous
under the assumption of $\cardd=\cardc$
(and more generally in any model of set theory
where $\betaoo$ has a P-point),
whether ZFC alone proves this inhomogeneity
is a significant open problem in its own right,
equivalent~\cite{mtuk} to Isbell's Problem:

\begin{question}
  Is it consistent with ZFC that $(\mc{U},\supset)$
  is Tukey equivalent to $([\cardc]^{<\alo},\subset)$
  for every free ultrafilter $\mc{U}$ on
  $\om$?~\cite{isbell,dobrinen}
\end{question}

This paper introduces pin equivalence, a strict strengthening
of Tukey equivalence of points in compact spaces.
Pin equivalence enjoys a representation theorem
in terms of closed binary relations
and, through Stone duality,
an appealing Boolean algebraic interpretation.
We will also show that, as in the case of topological homogeneity,
no infinite compact $F$-space is pin homogeneous.
On the other hand, as in the Tukey case,
$X\times 2^{\character{X}}$ is pin homogeneous for all compact $X$,
as is every first countable crowded compact $X$.
We will also show that $2^\om\times 2^{\oml}_{\tlex}$
is pin homogeneous despite having points with
different $\pi$-characters. Thus,
$2^\om\times 2^{\oml}_{\tlex}$ is a pin homogeneous
space with all points of Tukey type $\om\times\oml$,
which I count as a tiny bit of progress towards
answering a question I have asked before:

\begin{question}
  Is there a compact homogeneous space with
  points of Tukey type $\om\times\oml$?~\cite{mrect}
\end{question}

In every known example of a compact homogeneous space $X$,
the (neighborhood filters of) points are Tukey equivalent
to $([\character{X}]^{<\alo},\subset)$.~\cite{mnth}

Without further ado, pin equivalence defined:

\begin{defn}\label{pindef}\
  \begin{itemize}
  \item Call closed sets $A,B$ in a compact space $X$
    {\it pin equivalent} and write $A\pin B$
    if there exist a compact space $Y$,
    a continuous surjection $f\colon Y\to X$,\footnote{
    Between compact Hausdorff spaces,
    all continuous surjections are quotient maps.}
    and a homeomorphism $g\colon Y\to Y$
    such that $g\inv{f}A=\inv{f}B$.
  \item Call points $a,b$ in a compact space $X$
    pin equivalent and write $a\pin b$ if $\{a\}\pin\{b\}$.
  \item Call a compact space {\it pin homogeneous} if
    all pairs of points are pin equivalent.
  \end{itemize}  
\end{defn}
(Pin equivalence is transitive, but not obviously so.
Wait for the proof.)

Observe that, \wolog, $f$ may be assumed invertible
at $a$ and $b$ because we may replace $Y$ with
its quotient where $\inv{f}\{a\}$ and $\inv{f}\{b\}$
are collapsed to points. This leads to my motivation for ``pin.''
I visualize $X$ inflated to a continuous preimage $Y$,
but with $a$ and $b$ pinned down.

\begin{example}
  Closed intervals are pin homogeneous. To see why,
  let us show that $0\pin 2$ in $X=[-2,2]$.
  Let $A$ be the hollow diamond
  \[\{(x,y)\in[-2,2]^2\suchthat \abs{x}+\abs{y}=2\}.\]
  Truncate $A$ to $B=A\cap[-1,2]^2$ and then
  extend to $Y=B\cup[-2,-1]^2$. Then $f(x,y)=x$
  defines a continuous surjection from $Y$ to $X$ and
  $g(x,y)=(y,x)$ defines a continuous involution of $Y$
  such that \[g\inv{f}\{0\}=g\{(0,2)\}=\{(2,0)\}=\inv{f}\{2\}.\]
  We will show later that every instance of pin equivalence
  in an arbitrary compact $X$ is
  also witnessed by a symmetric subspace of $X^2$.
\end{example}

\begin{defn}
  A space is {\it Boolean} if is compact and
  has a base consisting of clopen sets.
\end{defn}

Restricting the definition of pin equivalence
to Boolean spaces and applying Stone duality,
we obtain and algebraic version of
pin equivalence that is very natural:
two filters are pin equivalent
if they generate isomorphic filters
in some larger Boolean algebra.
More precisely:

\begin{defn}\label{boolpindef}
  Given two filters $F,G$ of a Boolean algebra $A$,
  we say that $F$ and $G$ are {\it pin equivalent}
  in $A$ and write $F\pin G$ if there is
  a Boolean algebra $B$ extending $A$ and there is
  a (Boolean) automorphism $h$ of $B$ that
  sends the filter of $B$ generated by $F$
  to the filter of $B$ generated by $G$.
\end{defn}

Next observe that if $f$ in Definition~\ref{pindef}
is required to be a homeomorphism
instead of a mere continuous surjection, then
pin homogeneity becomes homogeneity.
This suggests a strategy for incremental progress
towards solving the open problem of
whether every compact space is a quotient
of a homogeneous compact space:
start with the positive solution to the
analogous problem for pin homogeneous compacts
and incrementally require more of $f$.

\begin{question}\label{increment}
  How much can we strengthen pin homogeneity
  before the analog of Question~\ref{bigquestion}
  for this intermediate homogeneity concept
  becomes as hard as Question~\ref{bigquestion} itself?
\end{question}

A natural strengthening of ``continuous surjection''
is ``open continuous surjection.''
So, let us define {\it open pin equivalence} and
{\it open pin homogeneity} by the requirement
that $f$ to also be an open map. 
Open pin homogeneous compact spaces appear
more difficult to obtain. 
In particular, open pin equivalence is easily seen
to preserve $\pi$-character.

\begin{question}
  Is every compact space a continuous image
  of an open pin homogeneous compact space?
\end{question}

\section{A representation theorem}
To my mind, the best evidence so far that pin equivalence
is worth studying is the following representation theorem.

\begin{thm}\label{represent}
  Points $a,b$ in compact space $X$ are pin equivalent
  iff there is a symmetric binary relation $R$ with
  domain $X$ such that $R$ is closed in $X^2$ and,
  for all $x\in X$, we have
  $aRx\biimp x=b$ and $bRx\biimp x=a$.
\end{thm}

\begin{cor}
  If $a,b$ are pin equivalent in a compact space $X$,
  then this is witnessed by $f,g,Y$ where
  $Y$ is a closed symmetric subset of $X^2$,
  $f\colon Y\to X$ is the first coordinate projection,
  and $g\colon Y\to Y$ is the continuous involution
  $(x,y)\mapsto(y,x)$.
\end{cor}

Before proving Theorem~\ref{represent},
we establish a version of the tube lemma for fibers.

\begin{lem}\label{fibertube}
  Suppose $Y$ is compact, $f\colon Y\to X$ is continuous,
  $x\in X$, and $V$ is a neighborhood of $\inv{f}\{x\}$.
  Then is a neighborhood $U$ of $x$ such that
  $\inv{f}U\subset V$.
\end{lem}
\begin{proof}
  Let $C$ be the complement of the interior of $V$,
  which is compact and disjoint from $\inv{f}\{x\}$.
  Then $fC$ is compact and disjoint from $\{x\}$.
  Let $U$ be the complement of $fC$,
  which is a neighborhood of $x$.
  Then $\inv{f}U$ is disjoint from $C$ and, therefore,
  a subset of the interior of $V$.
\end{proof}

\begin{defn}
  Given a binary relation $R$:
  \begin{itemize}
  \item Let $\inv{R}$ denote the converse relation.
  \item Given also a set $A$,
    let $RA$ denote the set of all $b$ such that
    $aRb$ for some $a\in A$.
  \item Given also a binary relation $S$,
    let $SR$ denote the set of all pairs
    $(a,c)$ such that $aRbSc$ for some $b$.
  \end{itemize}
\end{defn}

\begin{proof}[Proof of Theorem~\ref{represent}]
  Suppose $a\not=b$, $Y$ is compact,
  $f\colon Y\to X$ is a quotient map,
  $g\colon Y\to Y$ is a homeomorphism,
  $\{c\}=\inv{f}\{a\}$, $\{d\}=\inv{f}\{b\}$, and $g(c)=d$.
  Let us construct $R$. 

  Let $A_1,B_1$ be disjoint neighborhoods of $a,b$.
  Choose a neighborhood $C_2$ of $c$ such that
  $C_2\subset \inv{f}A_1$ and $g(C_2)\subset \inv{f}B_1$.
  Applying Lemma~\ref{fibertube}, choose a closed
  neighborhood $A_3$ of $a$ such that $\inv{f}A_3\subset C_2$.
  Let $B_3=fg\inv{f}A_3$, which is compact
  because $Y$ is compact. Applying Lemma~\ref{fibertube} again,
  $B_3$ is a neighborhood of $b$.
  Also, $A_3$ and $B_3$ are disjoint because
  \[A_3\subset fC_2\subset A_1\text{ and }
  B_3=fg\inv{f}A_3\subset fgC_2\subset B_1.\]

  Let $T$ be the all $(p,q,r)\in X\times X\times Y$
  such that $f(r)=p$ and $f(g(r))=q$. This set is compact.
  Let $S$ be the set of all $(p,q)\in A_3\times B_3$
  such that $(p,q,r)\in T$ for some $r\in Y$.
  This set is also compact. And, since $B_3=fg\inv{f}A_3$,
  the domain and range of $S$ are $A_3$ and $B_3$.
  Moreover, $a$ is the unique $p$ satisfying $pSb$ and
  $b$ is the unique $q$ satisfying $aSq$.
  Finally,
  let $D$ be the closure of the complement of $A_3\cup B_3$.
  Then $R=S\cup\inv{S}\cup D^2$ is as desired.
\end{proof}

\begin{defn}
  Given $a,b,X,R$ as in Theorem~\ref{represent},
  we say that $R$ {\it represents} $a\pin b$ in $X$.
\end{defn}

Our representation theorem helps us prove
several nice properties of pin equivalence,
starting with the next lemma, which we will use many times.
This lemma allows us to use a relation $R$ as above like
a function that is continuous at $a$ and $b$.

\begin{lem}\label{likects}
  If $R$ represents $a\pin b$ in compact space $X$
  and $V$ is a neighborhood of $b$, then
  $a$ has a neighborhood $U$ such that $RU\subset V$.
\end{lem}
\begin{proof}
  Suppose not. Then there are nets
  $(p_i)_{i\in I}$, $(q_i)_{i\in I}$ in $X$ such that
  $p_i\to a$, $p_iRq_i$, and $q_i\not\in V$.
  Since $X$ is compact, $(q_i)_{i\in I}$ has a cluster point $c$.
  And $c\not=b$ since $c$ is not in the interior of $V$.
  Since $R$ is closed, $aRc$. But this contradicts
  $aRx\biimp x=b$.
\end{proof}

\begin{cor}\label{Rnets}
  Suppose $R$ represents $a\pin b$
  and net $(x_i)_{i\in I}$ converges to $a$.
  If $x_iRy_i$ for all $i\in I$,
  then $(y_i)_{i\in I}$ converges to $b$.
\end{cor}

\begin{cor}\label{Ruf}
  Suppose, in a compact space,
  that $R$ represents $a\pin b$,
  $\mc{U}$ is an ultrafilter on a set $I$,
  and $\lim_{i\to\mc{U}}x_i=a$.
  If $x_iRy_i$ for all $i\in I$,
  then $\lim_{i\to\mc{U}}y_i=b$.
\end{cor}

\begin{thm}
  Pin equivalence is transitive.  
\end{thm}
\begin{proof}
  Suppose $a\pin b\pin c$ in compact space $X$.
  Let us show that $a\pin c$.
  \Wma\ $a,b,c$ are distinct.
  Let $R,S$ represent $a\pin b$ and $b\pin c$.
  Let $A_1,B_1,C_1$ be disjoint closed neighborhoods of $a,b,c$.
  Let $A_2,C_2$ be open neighborhoods of $a,c$ such that
  $A_2\subset A_1$, $C_2\subset C_1$,
  $RA_2\subset B_1$, and $SC_2\subset B_1$.
  Define relations $\hat{R},\hat{S},T$ as follows.
  \begin{align*}
    \hat{R}&=R\cap(X\setminus C_2)^2\\
    \hat{S}&=S\cap(X\setminus A_2)^2\\
    T&=\hat{R}\hat{S}\cup\hat{S}\hat{R}\cup B_1^2
  \end{align*}
  Then $T$ is symmetric, 
  \(aTx\biimp x=c\), and \(cTx\biimp x=a\).
  By compactness, $T$ is also closed.  
  We just need to show that $T$ has domain $X$.

  Fix $x\in X$. First, suppose $x\not\in A_1\cup B_1$.
  We then have $x\hat{S}y$ for some $y$. Moreover,
  $y\not\in C_2$ because $SC_2\subset B_1$ and $ySx$.
  Therefore, $y\hat{R}z$ for some $z$.
  Thus, $x\in\dom(\hat{R}\hat{S})$.
  If instead $x\not\in C_1\cup B_1$,
  then $x\in\dom(\hat{S}\hat{R})$
  by analogous reasoning.
  In the only remaining case, $x\in B_1$,
  we have $x\in\dom(B_1^2)$. Thus, $X=\dom(T)$.
\end{proof}

The theorem says that pin equivalence
is a local property.

\begin{thm}
  Let $Y,Z$ be closed subspaces of a compact space $X$.
  Suppose $a,b$ are in the interior of $Y\cap Z$
  and $a\pin b$ in $Y$. Then $a\pin b$ in $Z$.
\end{thm}
\begin{proof}
  Let $R$ represent $a\pin b$ in $Y$.
  Let $A_1,B_1$ be closed neighborhoods of $a,b$
  in $Y\cap Z$. Let $A_2,B_2$ be open neighborhoods of $a,b$
  such that $A_2\subset A_1$, $B_2\subset B_1$,
  $RA_2\subset B_1$, and $RB_2\subset A_1$.
  Let $S=R\cap(A_1\times B_1)$, $T=R\cap(B_1\times A_1)$,
  and $U=S\cup T\cup C^2$ where $C=Z\setminus(A_2\cup B_2)$.
  Then $U$ represents $a\pin b$ in $Z$.
\end{proof}

The next lemma isolates a recurring technique from the
proofs of the above theorems.

\begin{lem}\label{localpin}
  Given distinct $a,b$ in a compact space $X$, 
  there exists $R$ that represents $a\pin b$ iff
  there exists a closed binary relation $S$ on $X$
  such that $\dom(S)$ is a neighborhood of $a$,
  $\ran(S)$ is a neighborhood of $b$ disjoint from $\dom(S)$,
  $xSb\biimp x=a$, and $aSy\biimp y=b$.
\end{lem}
\begin{proof}
  Given $R$, let $S=R\cap((U\times RU)\cup(RV\times V))$
  for sufficiently small closed neighborhoods $U,V$ of $a,b$.
  Given instead $S$, let
  \[R=S\cup\inv{S}\cup\left(\closure{X\setminus
    (\dom(S)\cup\ran(S))}\right)^2.\qedhere\]
\end{proof}

\section{Pin equivalence vs. Tukey equivalence}
Here we show that pin equivalence
strictly implies Tukey equivalence.

\begin{defn}
  A {\it directed set} is a nonempty set $S$ equipped
  with a transitive reflexive relation $\leq$
  such that for all $x,y\in S$ there exists $z\in S$
  such that $x,y\leq z$.
\end{defn}

\begin{defn}
  Given two directed sets $P, Q$:
  \begin{itemize}
  \item 
    We say $P$ is {\it Tukey below} $Q$ and write $P\leqtuk Q$
    if there exists $f\colon Q\to P$ that is {\it convergent},
    that is, for every $p_0\in P$ there exists $q_0\in Q$
    such that $f(q)\geq p_0$ for all $q\geq q_0$.
  \item We say $P$ is {\it Tukey equivalent} to $Q$
    and write $P\eqtuk Q$ if $P\leqtuk Q\leqtuk P$.
  \item A subset $U$ of $P$ is {\it unbounded} if
    has no upper bound in $P$.
  \item A subset $C$ of $P$ is {\it cofinal} if
    for every $p\in P$ has an upper bound in $C$.
  \item The {\it cofinality} $\cf(P)$ of $P$ is
    least of the cardinalities of cofinal subsets of $P$.
  \item Given a cardinal $\ka$,
    we say $P$ is {\it $\ka$-directed} if
    every subset of $P$ size less than $\ka$
    has an upper bound in $P$.
  \item Given a cardinal $\ka$,
    we say $P$ is {\it $\ka$-OK} if, for each $f\colon\om\to P$
    there exists $g\colon\ka\to P$ such that
    for every $n<\om$,
    every increasing $n$-tuple $\xi_1<\cdots<\xi_n<\ka$,
    and every upper bound $b\in P$
    of $\{g(\xi_1),\ldots,g(\xi_n)\}$,
    we have $f(n)\leq b$.
  \end{itemize}
\end{defn}

Below are some elementary consequences of the above definitions.
\begin{itemize}
\item Composition preserves convergence.
\item If $C$ is a cofinal subset of $P$, then $C\eqtuk P$.
\item If $P\leqtuk Q$, then $\cf(P)\leq\cf(Q)$.
\item If $Q$ is $\ka$-directed and $P\leqtuk Q$,
  then $P$ is $\ka$-directed.
\item If $P$ is $\lm$-OK and $\ka\leq\lm$,
  then $P$ is $\ka$-OK.
\item If $\cf(P)\leq\ka$, then $P\leqtuk [\ka]^{<\alo}$
  where $[S]^{<\alo}$
  denotes the finite subsets of $S$
  ordered by inclusion $(\subset)$.
\end{itemize}

\begin{lem}\label{tukeytop}
  If $P$ is $\ka$-OK but not $\oml$-directed,
  then $[\ka]^{<\alo}\leqtuk P$.
\end{lem}
\begin{proof}
  Let $f$ map $\om$ to an unbounded subset of $P$.
  Let $g$ be as in the definition of $\ka$-OK.
  Then $g$ maps each infinite subset of $\ka$ to an unbounded set.
  To obtain a convergent map from $P$ to $[\ka]^{<\alo}$,
  map each $p\in P$ to the set of all
  $\xi<\ka$ satisfying $g(\xi)\leq p$.
\end{proof}

Through neighborhood filters, the order concepts
defined above induce the topological concepts defined next.

\begin{defn}
  Given points $a,b$ in space $X$:
  \begin{itemize}
  \item 
    We denote by $\mc{N}_X(a)$ the {\it neighborhood filter}
    of $a$, that is, the set of all $N\subset X$
    with $a$ in the interior of $N$.
    We make $\mc{N}_X(a)$ a directed set by
    ordering it by containment ($\supset$).
  \item 
    We say $a$ is Tukey below (resp., Tukey equivalent to)
    $b$ if $\mc{N}_X(a)\leqtuk\mc{N}_X(b)$
    (resp., $\mc{N}_X(a)\eqtuk\mc{N}_X(b)$).
  \item 
    We denote by $\character{a,X}$, the {\it character} of $a$,
    which is the cofinality $\cf(\mc{N}_X(a))$
    of $a$'s neighborhood filter.
  \item Given a cardinal $\ka$, we say $a$ is $\ka$-OK
    if its neighborhood filter is.
  \end{itemize}
\end{defn}

\begin{thm}
  If $a\pin b$ in compact space $X$, then $a\eqtuk b$.
\end{thm}
\begin{proof}
  Let $R$ represent $a\pin b$. \Ists\ $b\leqtuk a$.
  Define $r\colon\mc{N}_X(a)\to\mc{N}_X(b)$ by $r(U)=RU$.
  By Lemma~\ref{likects}, $r$ is convergent.
\end{proof}

\begin{defn}
  Given a point $a$ in a space $X$:
  \begin{itemize}
  \item We say $a$ is a {\it P-point}
    if $\mc{N}_X(a)$ is $\oml$-directed.
  \item We say $a$ is a {\it weak P-point}
    if $X\setminus C\in\mc{N}_X(a)$
    for every countable $C\subset X\setminus\{a\}$.
  \end{itemize}
\end{defn}

\begin{thm}
  In a compact space $X$, if $a\pin b$ and
  $a$ is not a weak P-point, then neither is $b$.
\end{thm}
\begin{proof}
  Let $a\in\closure{\{x_n\suchthat n<\om\}}$
  but $x_n\not=a$ for all $n<\om$.
  Letting some $R$ represent $a\pin b$,
  choose $y_n$ such that $x_nRy_n$, for each $n<\om$.
  Then $y_n\not=b$ for all $n<\om$.
  For each $N\in\mc{N}_X(a)$,
  choose $x_{\varphi(N)}\in N$,
  thus defining a net converging to $a$.
  Then $y_{\varphi(N)}\to b$ by Corollary~\ref{Rnets}.
  Hence, $b\in\closure{\{y_n\suchthat n<\om\}}$.
\end{proof}

Kunen proved that the \v Cech-Stone remainder $\om^*$
has weak P-points, a fact that previously
was merely known to be consistent with ZFC.
His proof consists of an easy result followed by a hard result:

\begin{lem}[Kunen~\cite{kunenok}]\label{omlOKweakP}
  In a space, if a point is $\oml$-OK,
  then it is also a weak P-point.
\end{lem}

\begin{lem}[Kunen~\cite{kunenok}]
  In the \v Cech-Stone remainder $\om^*$,
  there is a $\cardc$-OK point that is not a P-point.
\end{lem}

\begin{defn}
  We say a space $X$ is an {\it F-space}
  if every two disjoint open $F_\sigma$-sets
  have disjoint closures.
\end{defn}

$\om^*$ is the quintessential example of a compact F-space.
Indeed, a Stone space of Boolean algebra is an F-space
iff the algebra has the {\it countable separable property}:
every two countably generated ideals $I,J$ with
$I\cap J=\{0\}$ extend to principal ideals $I',J'$ with
$I'\cap J'=\{0\}$. Now $\om^*$ is homeomorphic to the
Stone space of $\powset(\om)/[\om]^{<\alo}$.
It is an easy exercise to show that this algebra
has the countable separable property.

\begin{thm}
  In the \v Cech-Stone remainder $\om^*$,
  there exist $a,b$ such that
  $a\eqtuk b$ but $a\not\pin b$.
\end{thm}
\begin{proof}
  Let $X=\om^*$.
  We identify each point $e\in X$ with the ultrafilter
  \[\{U\subset\om\suchthat e\in\closure{U}\}\]
  where the closure $\closure{U}$ is computed
  in the \v Cech-Stone compactification $\beta\om=\om\cup X$.
  The map $E\mapsto E\setminus\om$ surjects
  from the above ultrafilter to
  the set of the clopen neighborhoods of $e$.
  And for $U,V\subset\om$, we have
  $\closure{U}\setminus\om\subset\closure{V}\setminus\om$
  iff $U\subset^*V$
  where $\subset^*$ is inclusion modulo finite sets.
  Therefore, $e\eqtuk\mc{N}_X(e)$
  provided $e$ is ordered by $\supset^*$.

  Let $a\in X$ be $\cardc$-OK but not a P-point.
  Then $[\cardc]^{<\alo}\leqtuk a$ by
  Lemma~\ref{tukeytop}
  and $a$ is a weak P-point by Lemma~\ref{omlOKweakP}.
  Let $(c_n)_{n<\om}$ be a discrete sequence in $X$
  and let $b$ be the ultralimit $\lim_{n\to a}c_n$.
  Then $a\not\pin b$ because $b$ is not a weak P-point.

  \begin{claim}
    $a\leqtuk b$.
  \end{claim}
  \begin{proof}
  We will show that
  $\varphi(V)=\{n<\om\suchthat c_n\in\closure{V}\}$
  defines a convergent map from $(b,\supset^*)$ to
  $(a,\supset)$, noting that the identity map
  from $(a,\supset)$ to $(a,\supset^*)$ is convergent.
  Since $(c_n)_{n<\om}$ is discrete, there are
  disjoint open $F_\sigma$ sets $(O_n)_{n<\om}$ such that
  $c_n\in O_n$ for each $n$. Suppose $U\in a$.
  Since $X$ is an F-space, $\bigcup_{n\in U}O_n$
  and $\bigcup_{n\not\in U}O_n$ have disjoint closures.
  Since also $b\in\closure{\bigcup_{n\in U}O_n}$,
  we may choose $V_0\in b$ such that
  $\closure{V_0}$ is disjoint from $\{c_n\suchthat n\not\in U\}$.
  Therefore, $\varphi(V)\subset U$ for all $V\subset^* V_0$.
  \end{proof}
  
  Moreover, $b\leqtuk [\cardc]^{<\alo}$
  since $b$ has cardinality $\cardc$.
  Therefore, \[a\eqtuk b\eqtuk[\cardc]^{<\alo}.\qedhere\]
\end{proof}

\begin{rmrk}
  In the above proof, the justification of $a\leqtuk b$
  works for any $a\in\om^*$. It really shows that if
  $a$ is strictly below $b$ in the Rudin-Frol\'ik order,
  then $(a,\supset)$ is Tukey below $(b,\supset^*)$.
\end{rmrk}

\section{Pin inhomogeneity in other F-spaces}
Besides $\om^*$, another simply defined example of
a compact F-space is the absolute of $2^\om$,
that is, the Stone space $\Xi$ of the algebra
of regular open subsets of $2^\om$. This is an F-space
because any regular open algebra is complete.
On the other hand, $\Xi$ has a countable $\pi$-base
because $2^\om$ does. Therefore, $\Xi$ lacks weak P-points.
Hence, our construction of pin-inequivalent points in $\om^*$,
which relied on Kunen's construction of a weak P-point in $\om^*$,
cannot generalize to all infinite F-spaces.
Nevertheless, we can use a lemma from another paper of Kunen's
to show that every infinite F-space has pin-inequivalent points.

\begin{defn}Given ultrafilters $\mc{U},\mc{V}$ on $\om$,
  we say $\mc{U}$ is {\it Rudin-Keisler below} $\mc{V}$
  and write $\mc{U}\leqrk\mc{V}$ if there exists
  $f\colon\om\to\om$ such that $\beta f(\mc{V})=\mc{U}$,
  that is, such that $E\in\mc{U}$ iff $\inv{f}E\in\mc{V}$,
  for all $E\subset\om$.
\end{defn}

\begin{thm}[Kunen~\cite{kunenok}]
  There are Rudin-Keisler incomparable weak P-points in $\om^*$.
\end{thm}

\begin{lem}[Kunen~\cite{kunenlhc}]
  \label{rkweakp}
  Suppose $\mc{U},\mc{V}$ are Rudin-Keisler incomparable
  weak P-points in $\om^*$. Also suppose that,
  in a compact F-space $X$, $a$ is the $\mc{U}$-limit
  of a discrete $\om$-sequence. Then $a$ is not the
  $\mc{V}$-limit of any $\om$-sequence in $X\setminus\{a\}$.
\end{lem}

\begin{thm}
  Let $X$ be an infinite compact F-space.
  Then there exist $a,b\in X$ such that $a\not\pin b$.
\end{thm}
\begin{proof}
  Let $\mc{U},\mc{V}\in\om^*$ be Rudin-Keisler incomparable
  weak P-points.
  Let $(a_n)_{n<\om}$ be a discrete sequence in $X$,
  let $a=\lim_{n\to\mc{U}}a_n$, and let $b=\lim_{n\to\mc{V}}a_n$.
  By Lemma~\ref{rkweakp}, $b$ is not the $\mc{U}$-limit
  of any $\om$-sequence in $X\setminus\{b\}$.
  Seeking a contradiction, suppose that $R$ represents $a\pin b$.
  For each $n<\om$, choose $b_n$ such that $a_nRb_n$.
  Then $b_n\not=b$ for all $n<\om$. Also,
  by Corollary~\ref{Ruf}, $\lim_{n\to\mc{U}}b_n=b$.
  Thus, we have a contradiction.
\end{proof}

\section{Pin homogeneity}
\begin{defn}
  Given a point $a$ in a space $X$,
  a subset $\mc{S}$ of $\mc{N}_X(a)$ is a
  {\it neighborhood subbase} at $a$ if
  $\mc{N}_X(a)$ is the smallest filter containing $\mc{S}$.
\end{defn}
\begin{defn}
  A set $\mc{E}$ of sets is {\it independent} if,
  for each pair of finite nonempty $\mc{F},\mc{G}\subset\mc{E}$,
  if $\mc{F}\cap\mc{G}=\vn$, then
  $\bigcap\mc{F}\not\subset{\bigcup\mc{G}}$.
\end{defn}

\begin{lem}\label{indeppin}
  In a compact space $X$, if there is a bijection
  from an independent neighborhood subbase at $a$
  to an independent neighborhood subbase at $b$,
  then $a\pin b$.
\end{lem}
\begin{proof}
  \Wma\ $a\not=b$.
  Let $f_1\colon\mc{A}_1\to\mc{B}_1$ biject
  from an independent neighborhood subbase at $a$
  to an independent neighborhood subbase at $b$.
  First, we construct modified $f_1,\mc{A}_1,\mc{B}_1$
  for which $\bigcup\mc{A}_1$ and
  $\bigcup\mc{B}_1$ are disjoint.
  Choose finite nonempty $\mc{C}_1\subset\mc{A}_1$
  and $\mc{D}_1\subset\mc{B}_1$
  such that $\bigcap\mc{C}_1$ and
  $\bigcap\mc{D}_1$ are disjoint.
  Let $\mc{C}_2=\mc{C}_1\cup\inv{f_1}\mc{D}_1$
  and $\mc{D}_2=f_1\mc{C}_1\cup\mc{D}_1$.
  Let
  \begin{align*}
    \mc{A}_2&=\{U\cap\bigcap\mc{C}_2\suchthat
    U\in\mc{A}_1\setminus\mc{C}_2\};\\
    \mc{B}_2&=\{U\cap\bigcap\mc{D}_2\suchthat
    U\in\mc{B}_1\setminus\mc{D}_2\}.\\    
  \end{align*}
  Then $A=\bigcup\mc{A}_2$ and
  $B=\bigcup\mc{B}_2$ are disjoint. Moreover,
  because $\mc{A}_1$ and $\mc{B}_1$ are each independent,
  $\mc{A}_2$ and $\mc{B}_2$ are too and
  $f_2(U\cap\bigcap\mc{C}_2)=f_1(U)\cap\bigcap\mc{D}_2$
  defines a bijection from $\mc{A}_2$ to $\mc{B}_2$.
  
  For each $N\in\mc{A}_2$, let \[T_N=(N\times f_2(N))
  \cup((A\setminus N)\times(B\setminus f_2(N))).\]
  Because $\mc{A}_2$ and $\mc{B}_2$ are each independent,
  if $\mc{F}\subset\mc{A}_2$ is finite and nonempty,
  then $S_{\mc{F}}=\bigcap_{N\in\mc{F}}T_N$
  has domain $A$ and range $B$.
  By compactness, $\closure{S_{\mc{F}}}$
  has domain $\closure{A}$ and range $\closure{B}$;
  so does $S=\bigcap_{\mc{F}}\closure{S_{\mc{F}}}$.
  Moreover, if $xT_Nb$ for some $x,N$, then $x\in N$.
  Therefore, $xSb$ implies $x=a$.
  Likewise, $aSy$ implies $y=b$.
  By Lemma~\ref{localpin}, $a\pin b$.
\end{proof}

\begin{thm}\label{cantorpin}
  If $X$ is a compact space, $\ka$ is a cardinal,
  and $\character{x,X}\leq\ka$ for all $x\in X$,
  then $X\times 2^\ka$ is pin homogeneous.
\end{thm}
\begin{proof}
  Given $(a,b)\in X\times 2^\ka$, \istf\ an independent
  local subbase at $(a,b)$ of cardinality $\ka$. Let
  $\{A_\af\suchthat \af<\ka\}$ be a neighborhood subbase at $a$.
  For each $\af<\ka$, let
  \[U_\af=A_\af\times\{y\in 2^\ka\suchthat y(\af)=b(\af)\}.\]
  Then $\af\not=\bt\imp U_\af\not=U_\bt$ and
  $\mc{U}=\{U_\af\suchthat\af<\ka\}$
  is a local subbase at $(a,b)$.
  To see that $\mc{U}$ is independent, suppose
  $\sigma,\tau\in[\ka]^{<\alo}$ are disjoint.
  Define $c\in 2^\ka$ by $c(\af)=b(\af)$ iff $\af\in\sigma$.
  Then $(a,c)\in U_\af$ for all $\af\in\sigma$ and
  $(a,c)\not\in U_\af$ for all $\af\in\tau$.
\end{proof}

I was not able to adapt the above proof to
show that $X^\ka$ is pin homogeneous.

\begin{question}\label{powerpin}
  Does every compact space have a pin homogeneous power?
\end{question}

\begin{defn}
  A space is {\it crowded} if it has no isolated points.
\end{defn}

\begin{thm}\label{firstctbl}
  Suppose $X$ is a first countable crowded compact space.
  Then $X$ is pin homogeneous.
\end{thm}
\begin{proof}
  Let $a,b$ be distinct points in $X$.
  Let $\{A_n\suchthat n<\om\}$ and $\{B_n\suchthat n<\om\}$
  be neighborhood bases at $a$ and $b$ such that
  $A_n\supsetneq A_{n+1}$ and $B_n\supsetneq B_{n+1}$. Let
  \[S=\{(a,b)\}\cup\bigcup_{n<\om}\left(
  \closure{A_n\setminus A_{n+1}}
  \times\closure{B_n\setminus B_{n+1}}\right).\]
  By Lemma~\ref{localpin}, $a\pin b$.
\end{proof}

\begin{prop}
  Pin homogeneity is productive.
\end{prop}
\begin{proof}
  Suppose that for each $i$ in some set $I$
  we have $\xymatrix@C=1pc{Y_i\ar[r]^{g_i}&Y_i\ar[r]^{f_i}&X_i}$
  witnessing $a(i)\pin b(i)$ in $X_i$.
  Then, letting $X=\prod_iX_i$ and $Y=\prod_iY_i$,
  we have $\xymatrix@C=1pc{Y\ar[r]^{g}&Y\ar[r]^{f}&X}$
  witnessing $a\pin b$ in $X$ where
  $f(y)(i)=f_i(y(i))$ and $g(y)(i)=g_i(y(i))$.
\end{proof}

\section{Pin equivalence and Boolean algebras}
The proofs of Lemma~\ref{indeppin} and Theorem~\ref{firstctbl}
implicitly used Boolean isomorphisms between Boolean closures
of neighborhood bases. Thus, these results are actually
special cases of the following theorem.

\begin{defn}
  A {\it neighborhood subbase} of a subset $E$ of a space $X$
  is family $\mc{S}$ of subsets of $X$ such that smallest
  filter containing $\mc{S}$ is the set of neighborhoods of $E$.
\end{defn}

\begin{defn}
  Given a subset $E$ of a Boolean algebra $A$,
  $\gen{E}$ denotes the Boolean closure of $E$.
\end{defn}

\begin{thm}\label{boolisopin}
  Given closed disjoint subsets $H,K$ of a compact space $X$,
  we have $H\pin K$ if
  $H$ and $K$ have neighborhood subbases $\mc{U}$ and $\mc{V}$
  such that there is a map $f\colon\mc{U}\to\mc{V}$ that
  extends to a Boolean isomorphism
  $\varphi\colon\gen{\mc{U}}\to\gen{\mc{V}}$
  of the Boolean closures of $\mc{U}$ and $\mc{V}$
  in $\powset(X)$.
\end{thm}
\begin{proof}
  Let $\varphi\colon\gen{\mc{U}}\to\gen{\mc{V}}$ be as above.
  Choose $U\in\gen{\mc{U}}$ and $V\in\gen{\mc{V}}$ such
  that $H\subset U$, $K\subset V$, and $U\cap V=\vn$.
  Letting $A=U\cap\inv{\varphi}(V)$, we obtain
  $H\subset A$, $K\subset\varphi(A)$, and $A\cap\varphi(A)=\vn$.
  Let $\mc{A}$ be the Boolean subalgebra
  $\gen{\mc{U}}\cap\powset(A)$
  of $\powset(A)$ (not a Boolean subalgebra of $\powset(X)$);
  let $\mc{B}$ be the Boolean subalgebra
  $\gen{\mc{V}}\cap\powset(B)$
  of $\powset(B)$ where $B=\varphi(A)$;
  let $\psi$ be the restriction of $\varphi$ to $\mc{A}$.
  Then $\psi$ is a Boolean isomorphism to $\mc{B}$.

  For each finite partition $\mc{E}\subset\mc{A}$,
  the relation
  \[T_{\mc{E}}=\bigcup_{E\in\mc{E}}(E\times \psi(E))\]
  has domain $A$ and range $B$.
  Moreover, if $\mc{F}$ refines $\mc{E}$, then
  $T_{\mc{F}}\subset T_{\mc{E}}$. 
  By compactness, each $\closure{T_{\mc{E}}}$
  has domain $\closure{A}$ and range $\closure{B}$;
  so does $S=\bigcap_{\mc{E}}\closure{T_{\mc{E}}}$.
  For any $x$, if $xT_{\mc{E}}y$ for some $y\in K$,
  then $x\in E$ for the unique $E\in\mc{E}$
  with $H\subset E$. Therefore, $xSy\in K\imp x\in H$.
  Analogously, $H\ni xSy\imp y\in K$.
  Let \[R=S\cup\inv{S}\cup
  \left(\closure{X\setminus(A\cup B)}\right)^2.\]
  Then the involution $g\colon R\to R$ given by $g(x,y)=(y,x)$
  and the coordinate projection $f\colon R\to X$
  given by $f(x,y)=x$ witness that $H\pin K$.
\end{proof}

\begin{defn}\
  \begin{itemize}
  \item Two filters $F,G$ of a Boolean algebra
    are {\it incompatible} if $x\wedge y=0$
    for some $(x,y)\in F\times G$.
  \item A subset $E$ of a filter $F$ of a Boolean algebra $A$
    {\it generates} $F$ in $A$ is $F$ is the smallest filter
    of $A$ that contains $E$. 
  \end{itemize}  
\end{defn}

\begin{cor}\label{filteriso}
  Suppose that $F$ and $G$ are incompatible filters
  of a Boolean algebra and that they are generated
  by sets $D$ and $E$. If there is a map from
  $D$ to $E$ that extends to a Boolean isomorphism
  from $\gen{D}$ to $\gen{E}$, then $F\pin G$.
\end{cor}

When compared to Definition~\ref{boolpindef},
the converse of Corollary~\ref{filteriso} looks
too good to be true. But I have not yet found a counterexample.

\begin{problem}\label{isoconverse}
  Find a Boolean algebra with pin equivalent and incompatible
  filters $F$, $G$ such that for all bijections $\varphi\colon D\to E$,
  if $D$ generates $F$ and $E$ generates $G$, then $\varphi$ does
  not extend to a Boolean isomorphism from $\gen{D}$ to $\gen{E}$.
\end{problem}

We next use the above theorem to show that
pin equivalence does not preserve $\pi$-character.

\begin{defn}
  The {\it$\pi$-character} $\pi\character{a,X}$
  of a point $a$ in a space $X$ is the least of
  the cardinalities of families $\mc{F}$
  of nonempty open subsets of $X$ such that
  every neighborhood of $a$ contains an element of $\mc{F}$.
  Such a family is called a {\it local $\pi$-base} at $a$.
\end{defn}

\begin{defn}
  Given an ordinal $\af$, $2^\af_{\tlex}$ is the set
  of all $f\colon\af\to 2$ with the lexicographic ordering
  and the associated order topology (which is compact).
\end{defn}

In $L=2^{\oml}_{\tlex}$, every monotone $\om_2$-sequence
is eventually constant. But every point is
the limit of a strictly increasing $\oml$-sequence or
the limit of a strictly decreasing $\oml$-sequence (or both).
Moreover, topologically, there are exactly three types of points in $L$,
as shown in the illustration below.
\[
\begin{xymatrix}
  {\text{II}&\ar[l]^{\oml}
    \ar[r]_{\oml}&\text{I}&\ar[l]^{\oml}
    \ar[r]_{\om}&\text{III}&\ar[l]^{\oml}
    \ar[r]_{\oml}&\text{III}&\ar[l]^{\om}
    \ar[r]_{\oml}&\text{II}}
\end{xymatrix}
\]
$2^{\aleph_1}$-many points of $L$ are simultaneously
the limit of a strictly increasing $\oml$-sequence and
the limit of a strictly decreasing $\oml$-sequence.
Call these points {\it type I}.
Call the two endpoints and the $2^{\alo}$-many points of $L$
with either an immediate predecessor or immediate successor
{\it type II}.\footnote{
Types I and II are topologically distinguishable:
each type I point $a$ is in the closure of each of
two disjoint topological copies of $\oml$ in $L\setminus\{a\}$;
the type II points lack this property.}
All points of type I or II are P-points
with Tukey type $\oml$ and $\pi$-character $\oml$.
The remainder of $L$, the set of {\it type III} points,
consists of $2^{\alo}$-many limits of strictly increasing
or strictly decreasing $\om$-sequences.
These have Tukey type $\om\times\oml$
and have $\pi$-character $\om$ because
the nonempty open intervals
with endpoints from the $\om$-sequence form a local $\pi$-base.

In the product space $K=2^\om\times L$, there are no P-points.
But $K$ inherits both $\pi$-characters of $L$; indeed,
$\pi\character{(p,q),K}=\pi\character{q,L}$.
On the other hand, every point in $K$ has Tukey type $\om\times\oml$.
Interestingly, $K$ is also pin homogeneous.

\begin{defn}
  Subalgebras $A_0,\ldots A_{n-1}$ of a Boolean algebra $B$ are
  {\it independent} if, for all $x\in\prod_{i<n}A_i$,
  if $x(i)\not=0$ for all $i$, then $\bigwedge_{i<n}x(i)\not=0$.
\end{defn}

\begin{thm}
  $2^\om\times2^{\oml}_{\tlex}$ is pin homogeneous.
\end{thm}
\begin{proof}
  Continuing to use the above notation $K,L$, suppose that
  $x^0=(p^0,q^0)$, and $x^1=(p^1,q^1)$ are distinct points in $K$.
  For each $i<2$, conditionally define sets
  $P^i_n$, $Q^i_\af $, $R^i_n$, $S^i_\af$, $T^i_n$ as follows.
  Let $\{P^i_n\suchthat n<\om\}$ be a neighborhood base at $p^i$
  such that $P^i_n\supsetneq P^i_{n+1}$.
  If $q^i$ is type I or II, then let $(Q^i_\af)_{\af<\oml}$ be
  a sequence of intervals such that
  $Q^i_\af\supsetneq Q^i_\bt$ for $\af<\bt$,
  $q^i$ is in the interior of each $Q^i_\af$,
  and $\{q^i\}=\bigcap_\af Q^i_\af$.
  If $q^i$ is type III, let
  $(Q^i_\af)_{\af<\oml}$ be a sequence of rays of $L$ such that
  $Q^i_\af\supsetneq Q^i_\bt$ for $\af<\bt$ and
  $q^i$ is the interior of each $Q^i_\af$
  but on the boundary of $\bigcap_\af Q^i_\af$.
  In all cases, let $S^i_\af=2^\om\times Q^i_\af$.
  If $q^i$ is type I or II, let $T^i_n=P^i_n\times K$.
  If $q^i$ is type III, let $T^i_n=P^i_n\times R^i_n$
  where $(R^i_n)_{n<\om}$ is a sequence of rays of $L$
  such that $R^i_n\supsetneq R^i_{n+1}$ and
  $q^i$ is the interior of each $R^i_n$
  but on the boundary of $\bigcap_n R^i_n$.
  In all cases, \[\mc{B}^i=\{S^i_\af\suchthat \af<\oml\}
  \cup\{T^i_n\suchthat n<\om\}\] is a neighborhood subbase at $x^i$.
  
  Let $\mc{S}^i$ be the Boolean closure of the set of
  all sets of the form $S^i_\af$.
  Let $\mc{T}^i$ be the Boolean closure of the set of
  all sets of the form $T^i_n$.
  Let $\mc{U}^i=\gen{\mc{S}^i\cup\mc{T}^i}$.
  Since $S^i_\af\supsetneq S^i_\bt$ for $\af<\bt$,
  the map $S^0_\af\mapsto S^1_\af$ extends uniquely to
  an isomorphism of $\sigma\colon\mc{S}^0\to\mc{S}^1$.
  Likewise, the map $T^0_n\mapsto T^1_n$ extends uniquely to
  an isomorphism of $\tau\colon\mc{T}^0\to\mc{T}^1$.
  Moreover, for each $i<2$, $\mc{S}^i$ and $\mc{T}^i$
  are independent because, for each $\af<\oml$ and $n<\om$,
  the intersection of
  $S^i_\af\setminus S^i_{\af+1}$ and $T^i_n\setminus T^i_{n+1}$
  is nonempty because in all cases it contains
  \[(P^i_n\setminus P^i_{n+1})\times(Q^i_\af\setminus Q^i_{\af+1}).\]
  Therefore, $\sigma\cup\tau$ extends uniquely
  to an isomorphism from $\mc{U}^0$ onto $\mc{U}^1$.
  Therefore, by Theorem~\ref{boolisopin}, $x^0\pin x^1$.
\end{proof}

It is not too hard to generalize the above theorem
to $\prod_{m\leq n}X_m$ where $X_m=2^{\om_m}_{\tlex}$ and $n<\om$.
The points of this product space attain all $\pi$-characters
in $[\om,\om_n]$ and have Tukey type $\prod_{m\leq n}\om_m$.
(Note that the product and lexicographic order topologies
on $2^\om$ are identical.)

\begin{thm}
  For each $n<\om$, $\prod_{m\leq n}2^{\om_m}_{\tlex}$ is
  pin homogeneous.
\end{thm}
\begin{proof}
  For convenience, let $\om_{-1}=1$.
  Using the above $X_m$ notation,
  for each $m\leq n$ and $x\in X_m$ there is a least
  $s(x)\in\{-1,0,\ldots,m\}$ for which there are two
  strictly decreasing sequences of rays
  $(P_\af(x)\suchthat\af<\om_m)$ and
  $(Q_\bt(x)\suchthat\af<\om_{s(x)})$
  such that each of these rays has $x$ in its interior and
  \[\{x\}=\bigcap_{\af<\om_m}P_\af(x)\cap
  \bigcap_{\af<\om_{s(x)}}Q_\af(x).\]
  Given $y\in Y=\prod_{m\leq n}X_m$, \istst\ $y$ has a
  neighborhood subbase consisting of the union of $n+1$
  strictly decreasing chains $(R^m_\af \suchthat \af<\om_m)$
  for $m\leq n$ whose respective Boolean closures
  $\mc{A}_0,\ldots,\mc{A}_n$ are independent.
  Letting $y_i=y(i)$, $s_i=s(y_i)$, $P^i_\af=P_\af(y_i)$,
  and $Q^i_\af=Q_\af(y_i)$ for each $i\leq n$,
  define $R^m_\af=\prod_{i\leq n}S^{m,i}_\af$ where
  \[S^{m,i}_\af=
  \begin{cases}
    X_i&:i<m\\
    P^m_\af\cap Q^m_0&: i=m; s_m=-1\\
    P^m_\af&: i=m; 0\leq s_m<m\\
    P^m_\af\cap Q^m_\af&: i=m; s_m=m\\
    X_i&:i>m; s_i\not=m\\
    Q^i_\af&: i>m; s_i=m,\\
  \end{cases}\] thus making
  $(R^m_\af \suchthat \af<\om_m)$ strictly decreasing
  for each $m\leq n$ and the union of these $n+1$ chains
  a neighborhood subbase at $y$. Moreover,
  $\mc{A}_0,\ldots,\mc{A}_n$ are independent because
  if $\af(m)<\om_m$ for each $m\leq n$, then
  $\bigcap_{m\leq n}(R^m_{\af(m)}\setminus R^m_{\af(m)+1})$
  is nonempty because it contains
  $\prod_{m\leq n}(P^m_{\af(m)}\setminus P^m_{\af(m)+1})$.
  (In verifying this, a key observation is that
  \(P^m_{\af(m)}\setminus P^m_{\af(m)+1}
  =P^m_{\af(m)}\cap Q^m_\bt\setminus P^m_{\af(m)+1}\)
  for all $\bt$.)
\end{proof}

On the other hand, it is shown in~\cite{mrect} that
if $X$ is a compact space and $P$ and $Q$ are directed sets
such that $\cf(P),\cf(Q)\geq\om$ and $Q$ is $\cf(P)^{++}$-directed,
then $X$ has a point not of Tukey type $P\times Q$.
In particular, we cannot have a compact space,
pin homogeneous or otherwise, 
with all points of Tukey type $\om\times\om_2$.

\section{Acknowledgements}
  This research was conducted in part while I was an associate professor at Texas A\&M International University.


\begin{thebibliography}{9}
  \bibitem{vanmillsurvey}\textsc{A. \arhan, J. van Mill.} {\it Topological Homogeneity.} Recent Progress in General Topology III (K. Hart, J. van Mill, and P. Simon, eds.), Atlantis Press, 2014.
  \bibitem{dobrinen} \textsc{N. Dobrinen.} {\it Survey on the Tukey theory of ultrafilters.} Selected Topics in Combinatorial Analysis, Zbornik Radova, Mathematical Institutes of the Serbian Academy of Sciences, {\bf 17(25)} (2015), 53--80.
  \bibitem{dowpearl}\textsc{A. Dow, E. Pearl.} {\it Homogeneity in powers of zero-dimensional first-countable spaces.} Proc. Amer. Math. Soc. {\bf 125} (1997), 2503--2510. 
  \bibitem{isbell}\textsc{J. Isbell.} {\it The category of cofinal types II.} Trans. Amer. Math. Soc. {\bf 116} (1965), 394--416.
  \bibitem{kunenok} \textsc{K. Kunen.} {\it Weak P\nbd-points in $\mathbb{N}^*$}, Colloq. Math. Soc. J\'anos Bolyai {\bf 23} (1978), 741--749. (Reprinted in: {\it The Mathematical legacy of Eduard \v Cech} (M. Kat\v etov and P. Simon, eds.), Birkh\"auser Verlag (1993), 100--108.)
  \bibitem{kunenlhc} K. Kunen, \emph{Large homogeneous compact spaces}, Open Problems in Topology (J. van Mill and G. M. Reed, eds.), North-Holland Publishing Co., Amsterdam, 1990, pp. 261--270.
  \bibitem{vanmillopit2} \textsc{J. van Mill.} {\it Homogeneous compacta.} Open Problems in Topology II (E. Pearl, ed.), Elsevier, 2007, 189--195.
  \bibitem{mnth}\textsc{D. Milovich.} {\it Noetherian types of homogeneous compacta and dyadic compacta.} Topology Appl. {\bf 156} (2008), 443--464.
\bibitem{mtuk}\textsc{D. Milovich}, {\it Tukey classes of ultrafilters on $\omega$.} Topology Proceedings {\bf 32} (2008), 351--362.
  \bibitem{mrect}\textsc{D. Milovich.} {\it Forbidden rectangles in compacta.} Topology and its Applications {\bf 159} (2012), 3180--3189.
  \bibitem{motorov}\textsc{D. B. Motorov.} {\it On retracts of homogeneous bicompacta}. Vestnik MGU {\bf 5} (1985).
\end{thebibliography}
\end{document}